\documentclass{amsproc}

\usepackage{amssymb,amsmath,amsfonts,stmaryrd,mathrsfs,amscd}
\usepackage[sans]{dsfont}
\usepackage{fouridx,chemarrow}
\usepackage[dvips,cmtip,all]{xy}
\usepackage{setspace,paralist}
\usepackage[linkcolor=black]{hyperref}
\usepackage{lineno,color}
\usepackage{tikz,mathabx} 
\usetikzlibrary{matrix,arrows,decorations.pathmorphing}

\newtheorem{theorem}{Theorem}[section]
\newtheorem{lemma}[theorem]{Lemma}

\theoremstyle{definition}
\newtheorem{definition}[theorem]{Definition}

\newtheorem{proposition}[theorem]{Proposition}
\newtheorem{corollary}[theorem]{Corollary}
\newtheorem{example}[theorem]{Example}

\numberwithin{equation}{section}

\theoremstyle{remark}
\newtheorem{remark}[theorem]{Remark}

\numberwithin{equation}{section}

\def\Spec{\operatorname{Spec}}
\def\Min{\operatorname{Min}}
\def\Ass{\operatorname{Ass}}
\def\edim{\operatorname{edim}}

\def\rank{\operatorname{rank}}

\def\Fin{\operatorname{\varphi_{\varstar}^{\mathit{n}}}}

\def\Fim{\operatorname{\varphi_{\varstar}^{\mathit{m}}}}
\def\Finm{\operatorname{\varphi_{\varstar}^{\mathit{n}+\mathit{m}}}}
\def\Fi{\operatorname{\varphi_{\varstar}}}
\def\Ef{\operatorname{\mathit{f}_{\varstar}}}

\def\Fib{\operatorname{\overline{\varphi}_{\varstar}}}

\begin{document}

\title{Dynamics and entropy in local algebra}

\author{Mahdi Majidi-Zolbanin}
\address{Department of Mathematics, LaGuardia Community College of the City University of New York, 31-10 Thomson Avenue, Long Island City, NY 11101}
\email{mmajidi-zolbanin@lagcc.cuny.edu}

\author{Nikita Miasnikov}
\address{Department of Mathematics,The Graduate Center of the City University of New York, 365 Fifth Avenue, New York, NY 10016}
\email{n5k5t5@gmail.com}

%    Information for third author
\author{Lucien Szpiro}
\address{Department of Mathematics,The Graduate Center of the City University of New York, 365 Fifth Avenue, New York, NY 10016}
%\curraddr{}
\email{lszpiro@gc.cuny.edu}
\thanks{The second and third authors received funding from the NSF Grants DMS-0854746 and DMS-0739346.}

%    General info
%\subjclass[2000]{Primary 54C40, 14E20; Secondary 46E25, 20C20}

%\date{January 14, 2010.}

\dedicatory{}


\begin{abstract}
We introduce and study a notion of \emph{algebraic entropy} for self-maps of finite length of Noetherian local rings, and develop its properties. We show that it shares the standard properties of \emph{topological entropy}. For \emph{finite} self-maps we explore the connection between the degree of the map and its algebraic entropy, when the ring is a Cohen-Macaulay domain. As an application of algebraic entropy, we give a characteristic-free interpretation of the definition of Hilbert-Kunz multiplicity.
\end{abstract}

\maketitle

\section{Introduction and notations}\label{Introduction}
Iterating a map from a space to itself generates a \emph{discrete-time} dynamical system. One way to measure the complexity of such a system is by using the notion of \emph{entropy}. In dynamical systems entropy is a notion that ``measures the rate of increase in dynamical complexity as the system evolves with time.'' (\cite[p.~313]{Young}.) Depending on the type of the underlying space and its self-maps, different notions of entropy can be defined and studied. For example, for compact topological spaces and continuous self-maps, Adler, Konheim, and McAndrew introduced the notion of \emph{topological entropy} in~\cite{AdKoMc} and for probability spaces and measure-preserving self-maps, Kolmogorov introduced the notion of \emph{measure-theoretic entropy} in~\cite{Kol} and later Sinai improved this notion in~\cite{Sin}.  \par In this paper we are concerned with discrete-time dynamical systems generated by local self-maps (endomorphisms) of Noetherian local rings. For a Noetherian local ring \((R,\mathfrak{m})\), and a local self-map \(\varphi\) of \(R\) with the property that \(\varphi(\mathfrak{m})R\) is \(\mathfrak{m}\)-primary (a self-map with this property is said to be of finite length), we introduce a notion of \emph{algebraic entropy} and develop its properties. In particular, we show that this notion of algebraic entropy shares many properties of topological entropy. To describe a number of these properties, let \(X\) be a compact topological space with a continuous self-map \(\varphi\) and let \(h_{\mathrm{top}}(\varphi,X)\) be the topological entropy of \(\varphi\). (For definition of topological entropy see Appendix~\ref{appind}.) It is well-known that \(h_{\mathrm{top}}(\varphi,X)\) satisfies the following properties (see~\cite[p.p.~311-313]{AdKoMc}):\smallskip

\begin{compactenum}
\item[\textbf{1})] For any \(k\in\mathbb{N}\), \(h(\varphi^k,X)=k\cdot h(\varphi,X)\). 
\item[\textbf{2})] If \(Y\subset X\) is a closed \(\varphi\)-invariant subspace, then \(h(\varphi\restriction_Y,Y)\leq h(\varphi,X)\).
\item[\textbf{3})] If \(f:X\rightarrow X^\prime\) is a homeomorphism, then \(h(f\circ\varphi\circ f^{-1},X^\prime)=h(\varphi,X)\).
\item[\textbf{4})] If \(X=\bigcup_{i=1}^mY_i\), where \(Y_i\), (\(i=1,\ldots,m\)) are closed \(\varphi\)-invariant subsets, then \(h(\varphi,X)=\max\big\{h(\varphi\restriction_{Y_i},Y_i)\;|\: 1\leq i\leq m\big\}\).
\end{compactenum}\smallskip

\noindent Now let $\varphi$ be a local self-map of finite length of a Noetherian local ring $R$, and note that if $\mathfrak{a}$ is an $\varphi$-invariant ideal of $R$, that is, if $\varphi(\mathfrak{a})R\subseteq\mathfrak{a}$, then $\varphi$ induces a self-map of finite length of $R/\mathfrak{a}$. Denote this induced self-map by $\overline{\varphi}$. In Section~\ref{Fpro} we will show that algebraic entropy satisfies the following properties:\smallskip

\begin{compactenum}
\item[\textbf{1}$^\prime$)] For any \(k\in\mathbb{N}\), \(h_{\mathrm{alg}}(\varphi^k,R)=k\cdot h_{\mathrm{alg}}(\varphi,R)\). 
\item[\textbf{2}$^\prime$)] If \(\mathfrak{a}\) is an \(\varphi\)-invariant ideal, then \(h_{\mathrm{alg}}(\overline{\varphi},R/\mathfrak{a})\leq h_{\mathrm{alg}}(\varphi,R)\).
\item[\textbf{3}$^\prime$)] If \(f:R\rightarrow R^\prime\) is a ring isomorphism, then \(h_{\mathrm{alg}}(f\circ\varphi\circ f^{-1},R^\prime)=h_{\mathrm{alg}}(\varphi,R)\).
\item[\textbf{4}$^\prime$)] If all minimal prime ideals $\mathfrak{p}_1,\ldots,\mathfrak{p}_m$ of $R$ are \(\varphi\)-invariant, and if $\overline{\varphi}_i$ is the self-map induced by $\varphi$ on $R/\mathfrak{p}_i$, then $h_{\mathrm{alg}}(\varphi,R)=\max\big\{h_{\mathrm{alg}}(\overline{\varphi}_i,R/\mathfrak{p}_i)\;|\;1\leq i\leq m\big\}$.
\end{compactenum}\smallskip

The analogy between conditions \textbf{1} - \textbf{4} and \textbf{1$^\prime$} - \textbf{4$^\prime$} is evident.\par  A number of notions in local commutative algebra, that are associated with the Frobenius endomorphism and its iterations, currently only make sense in rings of positive characteristic (e.g., the Hilbert-Kunz multiplicity). The notion of algebraic entropy will allow us to give a characteristic-free interpretation of definitions of these notions. (In the case of the Hilbert-Kunz multiplicity, see Definition~\ref{HKM} and Remark~\ref{HKMA}.) Thus, using algebraic entropy, these notions may be defined and studied in any characteristic and for arbitrary self-maps of finite length.\par The organization of material in this paper is as follows: Section~\ref{Prelim} contains some preparatory but important facts about length of local ring homomorphisms, that will be used throughout the paper. We define the notion of algebraic entropy in Section~\ref{Defin} and give two examples. In Section~\ref{Firstpro} we develop the first set of properties of algebraic entropy. In Section~\ref{Fpro} we prove a number of properties of algebraic entropy, that are similar to properties of topological entropy. In Section~\ref{Findegr} we restrict our attention to \emph{finite} self-maps of domains and explore the connection between the degree and algebraic entropy of such self-maps. Finally, in Section~\ref{Bounn} we find upper and lower bounds for algebraic entropy. The Appendix contains a quick review of the definition of topological entropy, as defined in~\cite{AdKoMc}. \bigskip

\noindent \textbf{Notations.} All rings in this paper are assumed to be Noetherian, commutative and with identity element. By a self-map of a ring we mean an endomorphism of that ring. For a self-map \(\varphi\) of a ring we will write \(\varphi^n\) for the \(n\)-fold composition of \(\varphi\) with itself. Given a ring homomorphism \(f:R\rightarrow S\) and an \(S\)-module \(N\), we will denote by \(\Ef N\) the \(R\)-module obtained by restriction of scalars. That is, \(\Ef N\) is the \(R\)-module whose underlying abelian group is \(N\) and whose \(R\)-module structure is given by \(r\cdot x=f(r)\:x\), for \(r\in R\) and \(x\in\Ef N\). This notation is consistent with the one used in~\cite{Bourb}. 
If \(M\) is an \(R\)-module of finite length, we will denote its length by \(\ell_R(M)\). For a finite \(R\)-module \(M\) we will use \(\nu(M)\) to denote the minimum number of generators of \(M\) over \(R\). For a ring \(R\), the set of all minimal prime ideals of \(R\) will be denoted by \(\Min(R)\). Finally, if \(\varphi\) is a self-map of a ring \(R\), we will denote the self-map induced by \(\varphi\) on \(\Spec(R)\) by \({^a\varphi}\).
\section{Preliminaries}\label{Prelim}
\begin{definition}\label{deflambda}
A homomorphism \(f:(R,\mathfrak{m})\rightarrow(S,\mathfrak{n})\) of Noetherian local rings is said to be \emph{of finite length}, if it is local and \(f(\mathfrak{m})S\) is \(\mathfrak{n}\)-primary. In this case we define the length of $f$, \(\lambda(f)\in[1,\infty)\) as
\(\lambda(f):=\ell_S\big(S/f(\mathfrak{m})S\big)\).
\end{definition}
\begin{remark}
Every finite homomorphism is of finite length, but the converse is not true.
\end{remark}
\begin{proposition}\label{crucial}
Let \(f:(R,\mathfrak{m})\rightarrow(S,\mathfrak{n})\) be a homomorphism of finite length of Noetherian local rings. If \(\mathfrak{q}\) is an \(\mathfrak{m}\)-primary ideal of \(R\), then \(f(\mathfrak{q})S\) is \(\mathfrak{n}\)-primary.
\end{proposition}
\begin{proof}
Let \(\mathfrak{p}\) be a minimal prime ideal of \(f(\mathfrak{q})S\). Then
\[\mathfrak{q}\subset f^{-1}\left(f(\mathfrak{q})S\right)\subset f^{-1}(\mathfrak{p}).\]By assumption \(\mathfrak{q}\) is \(\mathfrak{m}\)-primary, so we must have \(f^{-1}(\mathfrak{p})=\mathfrak{m}\). But then \(\mathfrak{p}\supset f(\mathfrak{m})S\). Since \(f(\mathfrak{m})S\) is assumed to be \(\mathfrak{n}\)-primary, it is immediately clear that \(\mathfrak{p}=\mathfrak{n}\) and the result follows.
\end{proof}
\begin{corollary}\label{flc}
Suppose \(f:(R,\mathfrak{m})\rightarrow(S,\mathfrak{n})\) and \(g:(S,\mathfrak{n})\rightarrow(T,\mathfrak{p})\) are homomorphisms of finite length of Noetherian local rings. Then \(g\circ f\) is of finite length.
\end{corollary}
\begin{proof}
Since \(f\) is of finite length, \(f(\mathfrak{m})S\) is \(\mathfrak{n}\)-primary and since \(g\) is of finite length, by Proposition~\ref{crucial} \(g(f(\mathfrak{m})S)T\) is \(\mathfrak{p}\)-primary. Thus \((g\circ f)(\mathfrak{m})T\) is \(\mathfrak{p}\)-primary, that is, \(g\circ f\) is of finite length.
\end{proof}
\begin{corollary}\label{crucial4self}
Let \((R,\mathfrak{m})\) be a Noetherian local ring, let \(\varphi\) be a self-map of finite length of \(R\). Then \(\varphi^n\) is also of finite length for all \(n\geq1\).
\end{corollary}
\begin{proof}
We proceed by induction on \(n\). The statement is true for  \(n=1\) by assumption. Suppose the statement holds for \(n\), i.e., suppose \(\varphi^n\) is of finite length. Then since \(\varphi^{n+1}=\varphi^n\circ\varphi\), the result follows from Corollary~\ref{flc} and induction hypothesis.
\end{proof}
\begin{proposition}\label{product formula}
Let \(f:R\rightarrow S\) be a local homomorphism of Noetherian local rings with residue fields \(k_R\) and \(k_S\), respectively, and assume that \([\Ef k_S:k_R]<\infty\). If \(N\) is an \(S\)-module of finite length, then \(\Ef N\) is an \(R\)-module of finite length, and 
\begin{equation}\label{length behavior2}
\ell_R(\Ef N)=[\Ef k_S:k_R]\cdot\ell_S(N).
\end{equation}
\end{proposition}
\begin{proof}
Consider a composition series of \(N\), \(0=N_0\subsetneq N_1\subsetneq\ldots\subsetneq N_t=N\), 
with \(N_i/N_{i-1}\cong k_S\). This gives rise to a filtration of \(\Ef N\),
\[0=\Ef N_0\subsetneq \Ef N_1\subsetneq\ldots\subsetneq \Ef N_t=\Ef N,\]
with quotients \(\Ef N_i/\Ef N_{i-1}\cong \Ef k_S\). Thus, Equation~\ref{length behavior2} follows.
\end{proof}
\begin{proposition}\label{quotients}
Let \((R,\mathfrak{m})\) be a Noetherian local ring, and let \(\varphi\) be a self-map of finite length of \(R\). Let \(\mathfrak{a}\) be an ideal of \(R\) with the property \(\varphi(\mathfrak{a})R\subset\mathfrak{a}\). Let \(\overline{\varphi}\) be the local self-map induced by \(\varphi\) on \(R/\mathfrak{a}\). Then \(\overline{\varphi}\) is of finite length, and for any \(n\in\mathbb{N}\):
\[\lambda(\overline{\varphi}^{\:n})=\ell_{R/\mathfrak{a}}\Big(\frac{R/\mathfrak{a}}{[\varphi^n(\mathfrak{m})R+\mathfrak{a}]/\mathfrak{a}}\Big)=\ell_R\Big(\frac{R}{\varphi^n(\mathfrak{m})R+\mathfrak{a}}\Big).\]
\end{proposition}
\begin{proof}
The claim that \(\overline{\varphi}\) is of finite length quickly follows from the definition of finite length. From the commutative diagram 
\[\begin{tikzpicture} 
\matrix (m) [matrix of math nodes, row sep=2.8em, column sep=2.3em, text height=2ex, text depth=0.25ex] { R & R \\ R/\mathfrak{a} & R/\mathfrak{a}\\}; 
\path[->, font=\scriptsize]
(m-1-1) edge node[auto] {$ \varphi^n $} (m-1-2) 
              edge (m-2-1) 
(m-2-1) edge node[auto] {$ \overline{\varphi}^{\:n} $} (m-2-2)
(m-1-2) edge (m-2-2);
\end{tikzpicture}\]
we see immediately \(\overline{\varphi}^{\:n}(\mathfrak{m}/\mathfrak{a})(R/\mathfrak{a})=[\varphi^n(\mathfrak{m})R+\mathfrak{a}]/\mathfrak{a}\). Thus, as \((R/\mathfrak{a})\)-modules
\[\frac{R/\mathfrak{a}}{\overline{\varphi}^{\:n}(\mathfrak{m}/\mathfrak{a})(R/\mathfrak{a})}\cong\frac{R/\mathfrak{a}}{[\varphi^n(\mathfrak{m})R+\mathfrak{a}]/\mathfrak{a}}.\]
This establishes the first equality. For the second equality we apply Proposition~\ref{product formula} to the local homomorphism \(f:R\rightarrow R/\mathfrak{a}\), taking \(N\) to be 
\[\frac{R/\mathfrak{a}}{[\varphi^n(\mathfrak{m})R+\mathfrak{a}]/\mathfrak{a}}\:,\] and noting that \([\Ef k_{(R/\mathfrak{a})}:k_R]=1\), and that as \(R\)-modules
\[\frac{R/\mathfrak{a}}{[\varphi^n(\mathfrak{m})R+\mathfrak{a}]/\mathfrak{a}}\cong \frac{R}{\varphi^n(\mathfrak{m})R+\mathfrak{a}}.\]
\end{proof}
\begin{corollary}\label{genandleng}
Let \((R,\mathfrak{m},k)\) be a Noetherian local ring, and let \(\varphi\) be a \emph{finite} local self-map of \(R\). Then
\[\nu(\Fin R)=[\Fi k:k]^n\cdot\lambda(\varphi^n).\]
\end{corollary}
\begin{proof}
By Nakayama's Lemma
\[\nu(\Fin R)=\dim_k(\Fin R/\mathfrak{m}\Fin R)=\ell_R(\Fin R/\mathfrak{m}\Fin R)=\ell_R\left(\Fi(R/\varphi^n(\mathfrak{m})R)\right).\]
The result follows from Proposition~\ref{product formula} if we note that \([\Fin k:k]=[\Fi k:k]^n\).
\end{proof}
\begin{proposition}\label{flatchange}
Let \(f:(R,\mathfrak{m})\rightarrow(S,\mathfrak{n})\) be a homomorphism of finite length of Noetherian local rings. Let \(M\) be an \(R\)-module of finite length. Then\smallskip

\begin{compactenum}
\item[\textbf{a})] \(M\otimes_RS\) is an \(S\)-module of finite length.
\item[\textbf{b})] In general \(\ell_S(M\otimes_R S)\leq\lambda(f)\cdot\ell_R(M)\).
\item[\textbf{c})] If in addition \(f\) is \emph{flat}, then \(\ell_S(M\otimes_R S)=\lambda(f)\cdot\ell_R(M)\).
\end{compactenum}
\end{proposition}
\begin{proof}
We proceed by induction on \(\ell_R(M)\). If \(M\) is an \(R\)-module of length \(1\), then \(M\cong R/\mathfrak{m}\) and all three parts of the proposition hold because in that case
\[\ell_S\left(M\otimes_RS\right)=\ell_S\left((R/\mathfrak{m})\otimes_RS\right)=\ell_S\left(S/f(\mathfrak{m})S\right)=\lambda(f).\]
Now assume that all three parts of the proposition hold for all \(R\)-modules of length \(n_0\), and let \(M\) be an \(R\)-module with \(\ell_R(M)=n_0+1\). Let \(M^\prime\) be a submodule of \(M\) such that \(\ell_R(M^\prime)=1\) and \(\ell_R(M/M^\prime)=n_0\). Tensor the exact sequence \(0\rightarrow M^\prime\rightarrow M\rightarrow M/M^\prime\rightarrow0\) with \(S\) to obtain an exact sequence
\begin{equation}\label{someeq1}
M^\prime\otimes_RS\rightarrow M\otimes_RS\stackrel{\alpha}{\rightarrow} (M/M^\prime)\otimes_RS\rightarrow0,
\end{equation} 
As this sequence shows, there is a surjection \(M^\prime\otimes_RS\rightarrow\ker\alpha\rightarrow0\). Since \(M^\prime\otimes_RS\) is, by base of induction, of finite length as an \(S\)-module,  so is \(\ker\alpha\). As a result, \(M\otimes_RS\) is also an \(S\)-module of finite length, and the above sequence yields \(\ell_S(M\otimes_RS)\leq \ell_S(M^\prime\otimes_RS)+\ell_S\left((M/M^\prime)\otimes_RS\right)\). Now \textbf{a}) and \textbf{b}) quickly follow from this inequality by using the induction hypothesis.\par When \(f\) is flat, we can put a \(0\) on the left side of Sequence~\ref{someeq1}, and then we obtain the equality \(\ell_S(M\otimes_RS)=\ell_S(M^\prime\otimes_RS)+\ell_S\left((M/M^\prime)\otimes_RS\right)\). Now \textbf{c}) follows immediately from this equality by using the induction hypothesis.
\end{proof}
\begin{corollary}\label{proplambda}
Suppose \(f:(R,\mathfrak{m})\rightarrow(S,\mathfrak{n})\) and \(g:(S,\mathfrak{n})\rightarrow(T,\mathfrak{p})\) are homomorphisms of finite length of Noetherian local rings. Then with notation of Definition~\ref{deflambda},\smallskip 

\begin{compactenum}
\item[\textbf{a})] In general \(\lambda(g)\leq\lambda(g\circ f)\leq\lambda(g)\cdot\lambda(f)\).
\item[\textbf{b})] If in addition \(g\) is \emph{flat}, then \(\lambda(g\circ f)=\lambda(g)\cdot\lambda(f)\).
\end{compactenum}
\end{corollary}
\begin{proof}
\textbf{a}) By Corollary~\ref{flc} \(\lambda(g\circ f)<\infty\). Since \(g(f(\mathfrak{m})S)T\subset g(\mathfrak{n})T\), we see 
\[\ell_T\big(T/g(\mathfrak{n})T\big)\leq \ell_T\big(T/g(f(\mathfrak{m})S)T\big),\] that is, \(\lambda(g)\leq\lambda(g\circ f)\). For the second inequality we note that there is a canonical \(T\)-module isomorphism
\[\frac{T}{g\left(f(\mathfrak{m})\right)T}\cong\frac{S}{f(\mathfrak{m})S}\otimes_ST.\]
Thus, using Proposition~\ref{flatchange}-\textbf{b}
\begin{equation}\label{onemore}
\lambda(g\circ f)=\ell_T\Big(\frac{S}{f(\mathfrak{m})S}\otimes_ST\Big)\leq\lambda(g)\cdot\ell_S\Big(\frac{S}{f(\mathfrak{m})S}\Big)=\lambda(g)\cdot\lambda(f).
\end{equation}
\textbf{b}) If \(g\) is flat, then by Proposition~\ref{flatchange}-\textbf{c} the inequality in Equation~\ref{onemore} turns into an equality, and the result follows immediately.
\end{proof}
\section{Algebraic entropy of local self-maps}\label{Defin}
The lemma that follows is well-known in dynamical systems (see, for example~\cite[Theorem~4.9, p.~87]{Walters}).
\begin{lemma}\label{subadditively convergent}
Let \(\{a_n\}\) and \(\{b_n\}\) be sequences of real numbers which satisfy the following conditions:\smallskip

\noindent\begin{compactenum}
\item[\emph{\textbf{a})}] For all \(n\in\mathbb{N}\), \(a_n,b_n\geq1\) and \(\{\sqrt[n]{a_n}\}\) is bounded above;
\item[\emph{\textbf{b})}] For any \(n,m\in\mathbb{N}\), \(a_{n+m}\geq a_n\cdot a_m\), and \(b_{n+m}\leq b_n\cdot b_m\), respectively.
\end{compactenum}\smallskip 

\noindent Then the sequences \(\{(\log a_n)/n\}\) and \(\{(\log b_n)/n\}\) are both convergent. Moreover, 
\[\{(\log a_n)/n\}\rightarrow\sup_n\{(\log a_n)/n\}\ \ \mathrm{and}\ \ \{(\log b_n)/n\}\rightarrow \inf_n\{(\log b_n)/n\}.\]
\end{lemma}
\begin{proof}
Let \[\alpha:=\sup_n\:\left\{\frac{\log a_n}{n}\right\}\ \ \mathrm{(respectively}\ \ \beta:=\inf_n\:\left\{\frac{\log b_n}{n}\right\}\mathrm{)}.\] Note that by assumption \textbf{a}) \(\alpha\) is a non negative real number. Also, by assumption \textbf{b}) we have \(b_n\leq(b_1)^n\) for all \(n\geq1\). Thus \(\beta\) is also a non negative real number. Now for every \(\varepsilon>0\) there exists \(n_0\) such that \((\log a_{n_0})/n_0\geq\alpha-\varepsilon\) (respectively \((\log b_{n_0})/n_0\leq\beta+\varepsilon\)). Given an integer \(n>n_0\), let us write \(n=n_0q+r\), with \(0\leq r<n_0\). Then using our assumptions (c) and (a) we have
\[a_n\geq a_{n_0q}\cdot a_r\geq a_{n_0q}\geq (a_{n_0})^q\]
\[\mathrm{(respectively,\ }b_n\leq b_{n_0q}\cdot b_r\leq (b_{n_0})^q\cdot b_r.)\] From these inequalities we can deduce
\[\frac{\log a_n}{n}\geq\frac{qn_0}{n}\cdot\frac{\log a_{n_0}}{n_0}\geq\frac{qn_0}{n}\cdot(\alpha-\varepsilon)=\frac{n_0}{n_0+r/q}\cdot(\alpha-\varepsilon)\]
\[\mathrm{(respectively}\ \ \frac{\log b_n}{n}\leq\frac{qn_0}{n}\cdot\frac{\log b_{n_0}}{n_0}+\frac{\log b_r}{n}\leq\frac{qn_0}{n}\cdot(\beta+\varepsilon)+\frac{\log b_r}{n}=\]\[\frac{n_0}{n_0+r/q}\cdot(\beta+\varepsilon)+\frac{\log b_r}{n}\mathrm{)}.\]
Hence, if we take \(n\) large enough so that 
\[\frac{n_0}{n_0+r/q}\geq\frac{\alpha-2\varepsilon}{\alpha-\varepsilon}\ \ \mathrm{(respectively}\ \ \frac{\log b_r}{n}\leq\varepsilon\ \mathrm{for\ all}\ 0\leq r<n_0\mathrm{)},\] then we will get \((\log a_n)/n\geq(\alpha-2\varepsilon)\) (respectively, since \(\displaystyle{\frac{n_0}{n_0+r/q}}\leq1\), we will get
\[\frac{\log b_n}{n}\leq(\beta+\varepsilon)+\varepsilon=\beta+2\varepsilon.\mathrm{)}\] These inequalities together with the definition of \(\alpha\) and \(\beta\) show that
\[\lim_{n\rightarrow\infty}\frac{\log a_n}{n}=\alpha\ \ \mathrm{and}\ \ \lim_{n\rightarrow\infty}\frac{\log b_n}{n}=\beta.\]
\end{proof}
\begin{theorem}\label{main existence}
Let \((R,\mathfrak{m})\) be a Noetherian local ring, let \(\varphi\) be a self-map of finite length of \(R\). Then the sequence \(\{(\log\lambda(\varphi^n))/n\}\) converges to its infimum.
\end{theorem}
\begin{proof}
We apply Lemma~\ref{subadditively convergent} to prove this theorem, taking \(b_n=\lambda(\varphi^n)\). We need to verify that conditions of Lemma~\ref{subadditively convergent} hold.  It is clear that \(\lambda(\varphi^n)\geq1\). Moreover, the inequality \(\lambda(\varphi^{m+n})\leq \lambda(\varphi^m)\cdot \lambda(\varphi^n)\) holds by Corollary~\ref{proplambda}. Hence, by Lemma~\ref{subadditively convergent} the sequence \(\{(\log\lambda(\varphi^n))/n\}\) converges to its infimum.
\end{proof}
\begin{definition}
Let \((R,\mathfrak{m})\) be a Noetherian local ring, let \(\varphi\) be a self-map of finite length of \(R\). We define the \emph{algebraic entropy} of \(\varphi\) as 
\[h_{\mathrm{alg}}(\varphi,R):=\lim_{n\rightarrow\infty}\frac{\log\lambda(\varphi^n)}{n}.\]
\end{definition}
Calling a quantity entropy requires justification. In our case, notable analogies between $h_{\mathrm{alg}}(f,R)$ and topological entropy serve to justify our terminology. 
\begin{example}(Dimension zero)\label{dimzero}
Let \((R,\mathfrak{m})\) be a Noetherian local ring of dimension zero, and let \(\varphi\) be a local self-map of \(R\). Then \(R\) is Artinian and
\[1\leq\lambda(\varphi^n)\leq\ell(R)<\infty.\]
Applying logarithm, dividing by \(n\) and letting \(n\) approach infinity, we see that \(h_{\mathrm{alg}}(\varphi,R)=0\). Thus, the algebraic entropy of any local self-map of a Noetherian local ring of dimension zero is \(0\).
\end{example}
\begin{example}(Frobenius)\label{FrobEnd}
Let \((R,\mathfrak{m})\) be a Noetherian local ring of dimension \(d>0\) and of characteristic \(p>0\), and let \(\varphi\) be the Frobenius endomorphism of \(R\). Then, by~\cite[Proposition~3.2, p.~777]{Kunz1}
\[p^{nd}\leq\lambda(\varphi^n)\leq\min_{\{y_1,\ldots,y_d\}}\left[\ell_R\left(R/(y_1,\ldots,y_d)R\right)\right]\cdot p^{nd},\]
where \(\{y_1,\ldots,y_d\}\) runs over all systems of parameters of \(R\). Applying logarithm, dividing by \(n\) and letting \(n\) approach infinity we see that \(h_{\mathrm{alg}}(\varphi,R)=d\cdot\log p\). 
\end{example}
Following a number of ideas of Kunz, in~\cite{Monsk1} Monsky defined the Hilbert-Kunz multiplicity for the Frobenius endomorphism of Noetherian local rings of positive characteristic. He then showed that in this case, Hilbert-Kunz multiplicity always exists. It is not easy to compute Hilbert-Kunz multiplicities . It often takes on non-integer values. It is not even known whether it is always a rational number. Nevertheless, it has become evident through works of various authors, that it provides a reasonable measure of the singularity of \(R\). Inspired by Example~\ref{FrobEnd}, here we propose a characteristic-free interpretation of the definition of Hilbert-Kunz multiplicity. 
\begin{definition}[Hilbert-Kunz multiplicity]\label{HKM}
Let $R$ be a Noetherian local ring of arbitrary characteristic, and let $\varphi$ be a self-map of finite length of $R$. Let $$p(\varphi,R):=\big(\exp(h_{\mathrm{alg}}(\varphi,R))\big)^{1/\dim R}.$$ Then the Hilbert-Kunz multiplicity of \(R\) with respect to \(\varphi\) is defined as
\begin{equation}\label{HK}
e_{\mathrm{HK}}(\varphi,R):=\lim_{n\rightarrow\infty}\frac{\lambda(\varphi^n)}{p(\varphi,R)^{nd}},
\end{equation}
provided that the limit exists.
\end{definition}
\begin{remark}\label{HKMA}
In the general case, we do not know whether the limit in~\ref{HK} always exists or not. 
\end{remark}
\begin{proposition}
Let \((R,\mathfrak{m})\) be a Noetherian local ring, let \(\varphi\) be a self-map of finite length of \(R\). If \(\varphi(\mathfrak{m})R\neq\mathfrak{m}\), then \[\lim_{n\rightarrow\infty}\frac{1}{n}\cdot\log\ell_R\left(\mathfrak{m}/\varphi^n(\mathfrak{m})R\right)=h_{\mathrm{alg}}(\varphi,R).\]
\end{proposition}
\begin{proof}
From the exact sequence
\[0\rightarrow\mathfrak{m}/\varphi^n(\mathfrak{m})R\rightarrow R/\varphi^n(\mathfrak{m})R\rightarrow R/\mathfrak{m}\rightarrow0\]
we see that \[\ell_R\left(\mathfrak{m}/\varphi^n(\mathfrak{m})R\right)=\ell_R\left(R/\varphi^n(\mathfrak{m})R\right)-\ell_R\left(R/\mathfrak{m}\right)=\ell_R\left(R/\varphi^n(\mathfrak{m})R\right)-1.\]
Since \(\varphi(\mathfrak{m})R\neq\mathfrak{m}\), we have \(\lambda(\varphi^n)=\ell_R\left(R/\varphi^n(\mathfrak{m})R\right)\geq2\). Therefore
\[\frac{1}{2}\:\lambda(\varphi^n)\leq\lambda(\varphi^n)-1=\ell_R\left(\mathfrak{m}/\varphi^n(\mathfrak{m})R\right)\leq\lambda(\varphi^n).\]
We get the result by applying logarithm, dividing by \(n\) and letting \(n\) approach infinity.
\end{proof}
\begin{definition}[\cite{AvrIynMill06}, p.~80]\label{contracting}
A local self-map \(\varphi\) of a local ring \((R,\mathfrak{m})\) is said to be \emph{contracting}, if for each element \(x\in\mathfrak{m}\) the sequence \(\{\varphi^n(x)\}_{i\geq1}\) converges to \(0\) in the \(\mathfrak{m}\)-adic topology of \(R\).
\end{definition}
\begin{remark}[\cite{AvrIynMill06}, Lemma~12.1.4, p.~81]\label{alternative}
A local self-map \(\varphi\) of a Noetherian local ring \((R,\mathfrak{m})\) is contracting if and only if \(\varphi^{\edim(R)}(\mathfrak{m})\subset\mathfrak{m}^2\), where \(\edim(R)\) is the embedding dimension of \(R\).
\end{remark}
The following proposition shows that when \(\varphi\) is a contracting self-map in the sense of Definition~\ref{contracting}, then we can compute \(h_{\mathrm{alg}}(\varphi,R)\) using any ideal of definition of the ring.
\begin{proposition}\label{anyideal}
Let \((R,\mathfrak{m})\) be a Noetherian local ring, let \(\varphi\) be a \emph{contracting} self-map of finite length of \(R\). Let \(\mathfrak{q}\) be an \(\mathfrak{m}\)-primary ideal of \(R\). Then
\[\lim_{n\rightarrow\infty}\frac{1}{n}\cdot\log\ell\left(R/\varphi^n(\mathfrak{q})R\right)=h_{\mathrm{alg}}(\varphi,R).\]
\end{proposition}
\begin{proof}
Since \(\mathfrak{q}\) is \(\mathfrak{m}\)-primary and \(\varphi\) is contracting, by Remark~\ref{alternative} there is an integer \(s\) such that \(\varphi^s(\mathfrak{m})\subset\mathfrak{q}\subset\mathfrak{m}\). Applying \(\varphi^n\) and extending the results into ideals of \(R\) we obtain
\[\varphi^{n+s}(\mathfrak{m})R\subset\varphi^n(\mathfrak{q})R\subset\varphi^n(\mathfrak{m})R.\]Thus
\[\left(\lambda(\varphi^n)\right)^{1/n}\leq\left(\ell_R\left(R/\varphi^n(\mathfrak{q})R\right)\right)^{1/n}\leq\left(\lambda(\varphi^{n+s})\right)^{[1/(n+s)]\cdot[(n+s)/n]}.\]
We obtain the result by applying logarithm and letting \(n\) approach infinity.
\end{proof}
\section{General properties of algebraic entropy}\label{Firstpro}
\begin{lemma}\label{essence} 
Let \(f:(R,\mathfrak{m})\rightarrow(S,\mathfrak{n})\) be a flat homomorphism of finite length of Noetherian local rings. Let \(\varphi\) be a self-map of finite length of \(R\), and assume that \(\varphi\) can be extended to a self-map \(\psi\) of \(S\), in such a way that the following diagram commutes
\[\begin{tikzpicture} 
\matrix (m) [matrix of math nodes, row sep=2.7em, column sep=2.9em, text height=2ex, text depth=0.25ex] { R & R \\ S & S.\\}; 
\path[->, font=\scriptsize]
(m-1-1) edge node[auto] {$ \varphi $} (m-1-2) 
              edge node [left] {$ f $} (m-2-1) 
(m-2-1) edge node[auto] {$ \psi $} (m-2-2)
(m-1-2) edge node [auto] {$ f $}  (m-2-2);
\end{tikzpicture}\]
Then \(h_{\mathrm{alg}}(\varphi,R)=h_{\mathrm{alg}}(\psi,S)\).
\end{lemma}
\begin{proof}
We first show \(\lambda(\psi)<\infty\). The commutativity of the diagram shows that \(\psi\) must be a local map. By Corollary~\ref{flc} the map \(f\circ\varphi\) is of finite length. Then from the commutativity of the diagram it follows that \(\psi\circ f\) is also of finite length. Since \(\psi({\mathfrak{n}})S\supset\psi\left(f(\mathfrak{m})S\right)S\) and since \(\psi\circ f\) is of finite length, it follows that \(\psi({\mathfrak{n}})S\) is \(\mathfrak{n}\)-primary, that is, \(\psi\) is of finite length.\par By Corollary~\ref{crucial4self} \(\varphi^n\) and \(\psi^n\) are also of finite length. Noting that \(\psi^n\circ f=f\circ\varphi^n\) and using the flatness of \(f\), by Corollary~\ref{proplambda} we obtain
\begin{eqnarray}
\lambda(\varphi^n)=\lambda(f)\cdot\lambda(\varphi^n)/\lambda(f)&=&\lambda(f\circ\varphi^n)/\lambda(f)=\lambda(\psi^n\circ f)/\lambda(f)\leq\lambda(\psi^n)\nonumber\\ &\leq&\lambda(\psi^n\circ f)=\lambda(f\circ\varphi^n)=\lambda(f)\cdot\lambda(\varphi^n).\nonumber
\end{eqnarray}
Hence \(\lambda(\varphi^n)\leq\lambda(\psi^n)\leq\lambda(f)\cdot\lambda(\varphi^n)\). The result follows immediately from these inequalities by taking logarithms, dividing by \(n\), and letting \(n\) approach infinity.
\end{proof}
\begin{corollary}
Let \((R,\mathfrak{m})\) be a Noetherian local ring, and let \(\varphi\) be a self-map of finite length of \(R\). If \(\widehat{R}\) is the \(\mathfrak{m}\)-adic completion of \(R\) then
\[h_{\mathrm{alg}}(\varphi,R)=h_{\mathrm{alg}}(\widehat{\varphi},\widehat{R}).\]
\end{corollary}
\begin{lemma}\label{subring}
Let \((R,\mathfrak{m})\) be a Noetherian local ring and suppose \(\varphi: R\rightarrow R\) is a local self-map of \(R\). Then \(S:=\bigcap_{n\geq1}\varphi^n(R)\) is a local subring of \(R\) with maximal ideal \(\mathfrak{n}:=\bigcap_{n\geq1}\varphi^n(\mathfrak{m})\).  Furthermore, if \(\mathfrak{a}\) is the ideal generated by \(\mathfrak{n}\) in \(R\), then \(\varphi(\mathfrak{a})R\subseteq\mathfrak{a}\). If \(\varphi\) is additionally injective, then \(\varphi(\mathfrak{a})R=\mathfrak{a}\).
\end{lemma}
\begin{proof}
It is immediately clear that \(S\) is a subring of \(R\) and that \(\mathfrak{n}\) is an ideal of \(S\). To show that \(\mathfrak{n}\) is the (only) maximal ideal of \(S\), consider an element \(s\in S\setminus\mathfrak{n}\). Since \(s\not\in\mathfrak{n}\), there is an \(n_0\) such that \(s\not\in\varphi^{n_0}(\mathfrak{m})\). In fact, since for \(n\geq n_0\), \(\varphi^{n}(\mathfrak{m})\subseteq\varphi^{n_0}(\mathfrak{m})\), we see that \(s\not\in\varphi^n(\mathfrak{m})\) for all \(n\geq n_0\). Hence, there are \emph{units} \(y_n\in R\setminus\mathfrak{m}\) such that \(s=\varphi^n(y_n)\) for all \(n\geq n_0\). Since \(s\) is clearly a unit in \(R\), it has a \emph{unique} multiplicative inverse \(s^{-1}\) in \(R\). From uniqueness of multiplicative inverse it immediately follows that we must have \(s^{-1}=\varphi^n(y_n^{-1})\), for all \(n\geq n_0\). Hence, \(s^{-1}\in S\), that is, \(s\) is also a unit in \(S\).\par
To prove the statements about the ideal \(\mathfrak{a}\), note that by its definition, \(\mathfrak{a}\) has a set of generators \(x_1,\ldots,x_g\in\mathfrak{n}\). So \(\varphi(\mathfrak{a})R\) can be generated by \(\varphi(x_1),\ldots,\varphi(x_g)\) and it suffices to show that each \(\varphi(x_i)\) is in \(\mathfrak{a}\). Since \(x_i\in\mathfrak{n}\), there is a sequence of element \(y_{i,n}\in\mathfrak{m}\) such that \(x_i=\varphi(y_{i,1})=\ldots=\varphi^n(y_{i,n})=\ldots.\) Thus,
\(\varphi(x_i)=\varphi^2(y_{i,1})=\ldots=\varphi^{n+1}(y_{i,n})=\ldots\), showing that \(\varphi(x_i)\in\mathfrak{n}\subset\mathfrak{a}\).
\par Now assume that \(\varphi\) is injective. To show \(\varphi(\mathfrak{a})R=\mathfrak{a}\) it suffices to show that each \(x_i\) is in \(\varphi(\mathfrak{a})\). Since \(x_i\in\mathfrak{n}\), there is a sequence of element \(y_{i,n}\in\mathfrak{m}\) such that
\[x_i=\varphi(y_{i,1})=\ldots=\varphi^n(y_{i,n})=\ldots.\] Since \(x_i=\varphi(y_{i,1})\), we will be done by showing that \(y_{i,1}\in\mathfrak{n}\). By injectivity of \(\varphi\), we have 
\(y_{i,1}=\varphi(y_{i,2})=\ldots=\varphi^{n-1}(y_{i,n})=\ldots,\) which means \(y_{i,1}\in\mathfrak{n}\).
\end{proof}
\begin{remark}\label{surp}
Let \((R,\mathfrak{m})\) be a Noetherian local ring and suppose \(\varphi: R\rightarrow R\) is a local self-map of \(R\). If \(\bigcap_{n\geq1}\varphi^n(\mathfrak{m})=(0)\), then we see from Lemma~\ref{subring} that \(R\) contains a field and is equicharacteristic. As noted in~\cite[Remark~5.9, p.~10]{HA}, this occurs, for example, if \(\varphi\) is a contracting self-map (in the sense of Definition~\ref{contracting}).
\end{remark}
The next lemma shows that  any self-map of a local ring of mixed characteristic naturally induces a self-map of another local ring of equicharacteristic \(p>0\) with the same algebraic entropy.
\begin{lemma}
Let \((R,\mathfrak{m})\) be a Noetherian local ring, and let \(\varphi\) be a self-map of finite length of \(R\). Let \(\mathfrak{a}\) be the ideal of \(R\) defined in Lemma~\ref{subring}, and let \(\overline{\varphi}\) be the local self-map induced by \(\varphi\) on \(R/\mathfrak{a}\) \emph{(}see Lemma~\ref{subring}\:\emph{)}. Then\smallskip

\begin{compactenum}
\item[\emph{\textbf{a})}] \(R/\mathfrak{a}\) is a local ring of equicharacteristic \(p>0\).
\item[\emph{\textbf{b})}] \(h_{\mathrm{alg}}(\overline{\varphi},R/\mathfrak{a})=h_{\mathrm{alg}}(\varphi,R)\). 
\end{compactenum}
\end{lemma}
\begin{proof}
\textbf{a}) With reference to Lemma~\ref{subring}, the image of the local subring \(S\) of \(R\) in \(R/\mathfrak{a}\) is a field, because it's maximal ideal \(\mathfrak{n}\) is contained in \(\mathfrak{a}\) and is mapped to \(0\). Hence \(R/\mathfrak{a}\) contains a field and must be a local ring of equicharacteristic \(p>0\), as its residue field is of characteristic \(p>0\).  \par \textbf{b}) Note that \(\varphi^n(\mathfrak{m})R\supset\mathfrak{a}\) for all \(n\geq1\). Hence, \(\varphi^n(\mathfrak{m})R+\mathfrak{a}=\varphi^n(\mathfrak{m})R\). Thus by Proposition~\ref{quotients}, \(\lambda(\overline{\varphi}^{\:n})=\lambda(\varphi^n)\). Our claim quickly follows from this equality.
\end{proof}
\begin{proposition}\label{genconv}
Let \(R\) be a Noetherian ring, and let \(\varphi\) be a \emph{finite} local self-map of \(R\). If we denote the minimum number of generators of the \(R\)-module \(\Fin R\) by \(\nu(\Fin R)\), then the sequence \(\{(\log \nu(\Fin R))/n\}\) converges to its infimum. We will denote this limit by by \(\nu_{\infty}\).
\end{proposition}
\begin{proof}
We will apply Lemma~\ref{subadditively convergent} to prove this proposition, taking \[b_n=\nu(\Fin R).\]
To verify conditions of Lemma~\ref{subadditively convergent}, first note that the inequality
\(b_{n+m}\leq b_n\cdot b_m\) holds because if \(\{x_1,\ldots,x_t\}\) and \(\{y_1,\ldots,y_s\}\) are sets of generators of \(\Fim R\) and \(\Fin R\) over \(R\), respectively, then \(\{\varphi^m(y_j)x_i\mid 1\leq i\leq t, 1\leq j\leq s\}\) is a set of generators of \(\Finm R\) over \(R\). Therefore
\[\nu(\Finm R)\leq\nu(\Fin R)\cdot\nu(\Fim R).\] On the other hand, it is clear that \(b_n=\nu(\Fin R)\geq1\). Hence, by Lemma~\ref{subadditively convergent} the sequence  
\(\{(\log \nu(\Fin R))/n\}\) converges to its infimum.
\end{proof}
\begin{corollary}
Let \((R,\mathfrak{m},k)\) be a Noetherian local ring, and let \(\varphi\) be a \emph{finite} local self-map of \(R\). Then \(\nu_{\infty}=\log[\Fi k:k]+h_{\mathrm{alg}}(\varphi,R)\), where \(\nu_{\infty}\) is as defined in Proposition~\ref{genconv}.
\end{corollary}
\begin{proof}
By Corollary~\ref{genandleng} \(\nu(\Fin R)=[\Fi k:k]^n\cdot\lambda(\varphi^n)\). The result follows by applying logarithm to both sides of this equation, then dividing by \(n\) and letting \(n\) approach infinity.
\end{proof}

\section{Entropy-type properties of algebraic entropy}\label{Fpro}
\begin{proposition}\label{exponent}
Let \((R,\mathfrak{m})\) be a Noetherian local ring, let \(\varphi\) be a self-map of finite length of \(R\). Then for any \(k\in\mathbb{N}\) \[h_{\mathrm{alg}}(\varphi^k,R)=k\cdot h_{\mathrm{alg}}(\varphi,R).\]
\end{proposition}
\begin{proof}
By definition of algebraic entropy 
\begin{eqnarray}
h_{\mathrm{alg}}(\varphi^k,R)&=&\lim_{n\rightarrow\infty}(1/n)\cdot\log\lambda(\varphi^{kn})\nonumber\\
&=&k\cdot\lim_{n\rightarrow\infty}(1/(kn))\cdot\log\lambda(\varphi^{kn})\nonumber\\
&=&k\cdot h_{\mathrm{alg}}(\varphi,R).\nonumber
\end{eqnarray}
\end{proof}
We need the following two lemmas for the proof of Theorem~\ref{irreducible components}.
\begin{lemma}[see~\cite{AdKoMc}, p.~312]\label{maximum limit}
Let \(\{a_n\}\) and \(\{b_n\}\) be two sequences of real numbers not less than \(1\) such that \(\lim_{n\rightarrow\infty}(\log a_n)/n=\alpha\) and \(\lim_{n\rightarrow\infty}(\log b_n)/n=\beta\) exist. Then \[\lim_{n\rightarrow\infty}\log(a_n+b_n)/n=\max\{\alpha,\beta\}.\]
\end{lemma}
\begin{lemma}\label{combine}
Let \((R,\mathfrak{m})\) be a Noetherian local ring, let \(\varphi\) be a self-map of finite length of \(R\). Let \(\mathfrak{a}_1,\ldots,\mathfrak{a}_s\) be a collection of not necessarily distinct ideals of \(R\), such that for each \(\mathfrak{a}_i\), \(\varphi(\mathfrak{a}_i)R\subset\mathfrak{a}_i\). Let \(\overline{\varphi}\) and \(\overline{\varphi}_i\) be the self-maps induced by \(\varphi\) on \(R/{\scriptstyle\prod}_i\mathfrak{a}_i\) and \(A/\mathfrak{a}_i\), respectively. Then
\[h_{\mathrm{alg}}(\overline{\varphi},R/{\scriptstyle\prod}_i\mathfrak{a}_i)=\max\{h_{\mathrm{alg}}(\overline{\varphi}_i,R/\mathfrak{a}_i)\mid 1\leq i\leq s\}.\] 
\end{lemma}
\begin{proof}
We proceed by induction on \(s\), the number of ideals, counting possible repetitions. The statement is trivially true if \(s=1\), so suppose \(s=2\). Without loss of generality we may assume
\[h_{\mathrm{alg}}(\overline{\varphi}_1,R/\mathfrak{a}_1)=\max\{h_{\mathrm{alg}}(\overline{\varphi}_1,R/\mathfrak{a}_1),h_{\mathrm{alg}}(\overline{\varphi}_2,R/\mathfrak{a}_2)\}.\]
Since \(\mathfrak{a}_1\mathfrak{a}_2\subset\mathfrak{a}_1\), we have \(\mathfrak{a}_1\cap\left(\mathfrak{a}_1\mathfrak{a}_2+\varphi^n(\mathfrak{m})R\right)=\mathfrak{a}_1\mathfrak{a}_2+(\mathfrak{a}_1\cap\varphi^n(\mathfrak{m})R)\). Thus, if we apply the Second Isomorphism Theorem to make the identification
\[\frac{\mathfrak{a}_1+\varphi^n(\mathfrak{m})R}{\mathfrak{a}_1\mathfrak{a}_2+\varphi^n(\mathfrak{m})R}\cong\frac{\mathfrak{a}_1}{\mathfrak{a}_1\mathfrak{a}_2+(\mathfrak{a}_1\cap\varphi^n(\mathfrak{m})R)},\] then we can write an exact sequence
\[0\rightarrow\frac{\mathfrak{a}_1}{\mathfrak{a}_1\mathfrak{a}_2+(\mathfrak{a}_1\cap\varphi^n(\mathfrak{m})R)}\rightarrow\frac{R}{\mathfrak{a}_1\mathfrak{a}_2+\varphi^n(\mathfrak{m})R}\rightarrow\frac{R}{\mathfrak{a}_1+\varphi^n(\mathfrak{m})R}\rightarrow0.\]
From this exact sequence we obtain
\begin{eqnarray}\label{lengthineq2}
\ell_R(R/[\mathfrak{a}_1+\varphi^n(\mathfrak{m})R])&\leq&\ell_R(R/[\mathfrak{a}_1\mathfrak{a}_2+\varphi^n(\mathfrak{m})R])\nonumber\\ 
& =&\ell_R(\mathfrak{a}_1/[\mathfrak{a}_1\mathfrak{a}_2+\left(\mathfrak{a}_1\cap\varphi^n(\mathfrak{m}\right)R)])\\
& + & \ell_R(R/[\mathfrak{a}_1+\varphi^n(\mathfrak{m})R]).\nonumber
\end{eqnarray}
Since in the quotient ring \(R/(\mathfrak{a}_1\mathfrak{a}_2)\) the ideal \(\mathfrak{a}_2/(\mathfrak{a}_1\mathfrak{a}_2)\) annihilates \(\mathfrak{a}_1/(\mathfrak{a}_1\mathfrak{a}_2)\), we can consider \(\mathfrak{a}_1/(\mathfrak{a}_1\mathfrak{a}_2)\) as a finite \(\left[\left(R/(\mathfrak{a}_1\mathfrak{a}_2)\right)/\left(\mathfrak{a}_2/(\mathfrak{a}_1\mathfrak{a}_2)\right)\right]\)-module and as such, there is a surjection
\[\left(\frac{R/(\mathfrak{a}_1\mathfrak{a}_2)}{\mathfrak{a}_2/(\mathfrak{a}_1\mathfrak{a}_2)}\right)^t\rightarrow\frac{\mathfrak{a}_1}{(\mathfrak{a}_1\mathfrak{a}_2)}\rightarrow0.\]
If we tensor this surjection, over the quotient ring \(R/(\mathfrak{a}_1\mathfrak{a}_2)\) with \[\frac{R/(\mathfrak{a}_1\mathfrak{a}_2)}{[\mathfrak{a}_1\mathfrak{a}_2+\varphi^n(\mathfrak{m})R]/(\mathfrak{a}_1\mathfrak{a}_2)}\] and then compare the lengths in the resulting surjection, by using Proposition~\ref{quotients}, Proposition~\ref{product formula} and the Third Isomorphism Theorem, we can quickly see
\begin{eqnarray}
\ell_R(\mathfrak{a}_1/[\mathfrak{a}_1\mathfrak{a}_2+\mathfrak{a}_1\cdot\varphi^n(\mathfrak{m})R])&\leq&\ell_R(\mathfrak{a}_1/[\mathfrak{a}_1^2\mathfrak{a}_2+\mathfrak{a}_1\cdot\varphi^n(\mathfrak{m})R])\nonumber\\
&\leq&t\cdot\ell_R(R/[\mathfrak{a}_2+\varphi^n(\mathfrak{m})R]).\nonumber
\end{eqnarray}
Since \(\ell_R(\mathfrak{a}_1/[\mathfrak{a}_1\mathfrak{a}_2+(\mathfrak{a}_1\cap\varphi^n(\mathfrak{m})R)])\leq\ell_R(\mathfrak{a}_1/[\mathfrak{a}_1\mathfrak{a}_2+\mathfrak{a}_1\cdot\varphi^n(\mathfrak{m})R])\), the previous inequality together with Inequality~\ref{lengthineq2} yield
\begin{eqnarray}
\ell_R(R/[\mathfrak{a}_1+\varphi^n(\mathfrak{m})R])&\leq&\ell_R(R/[\mathfrak{a}_1\mathfrak{a}_2+\varphi^n(\mathfrak{m})R])\nonumber \\
 &\leq& \ell_R(R/[\mathfrak{a}_1+\varphi^n(\mathfrak{m})R])+t\cdot\ell_R(R/[\mathfrak{a}_2+\varphi^n(\mathfrak{m})R]).\nonumber  
\end{eqnarray}
If we apply logarithm, divide by \(n\), and let \(n\) approach infinity, by using Lemma~\ref{maximum limit} and Proposition~\ref{quotients} we obtain
\[h_{\mathrm{alg}}(\overline{\varphi}_1,R/\mathfrak{a}_1)\leq h_{\mathrm{alg}}(\overline{\varphi},R/\mathfrak{a}_1\mathfrak{a}_2)\leq\max\{h_{\mathrm{alg}}(\overline{\varphi}_1,R/\mathfrak{a}_1),h_{\mathrm{alg}}(\overline{\varphi}_2,R/\mathfrak{a}_2)\}.\]  This establishes the result for \(s=2\). Now we assume the statement holds for all \(s\) with \(2\leq s\leq n_0\), and we show it also holds for \(s=n_0+1\). To this end, we can write the product \({\scriptstyle\prod}_{i=1}^{n_0+1}\:\mathfrak{a}_i\) of our ideals in the form \(({\scriptstyle\prod}_{i=1}^{n_0}\:\mathfrak{a}_i)(\mathfrak{a}_{n_0+1})\) and then apply the case \(s=2\) followed by the case \(s=n_0\) to establish the result for \(s=n_0+1\), using the induction hypothesis.
\end{proof}
The next result shows that if a self-map \(\varphi\) of a Noetherian local ring \(R\) `fixes' minimal prime ideals of \(R\), then its algebraic entropy is equal to the maximum algebraic entropy of the self-maps that it induces on each irreducible component of \(\Spec(R)\). 
\begin{theorem}\label{irreducible components}
Let \((R,\mathfrak{m})\) be a Noetherian local ring, let \(\varphi\) be a self-map of finite length of \(R\). Assume further that for every minimal prime ideal \(\mathfrak{p}_i\) of \(R\) we have \(\varphi(\mathfrak{p}_i)R\subset\mathfrak{p}_i\) and let \(\overline{\varphi}_i\) be the self-map induced by \(\varphi\) on \(R/\mathfrak{p}_i\). Then
\begin{equation}\label{irrcomps}
h_{\mathrm{alg}}(\varphi,R)=\max\{h_{\mathrm{alg}}(\overline{\varphi}_i,R/\mathfrak{p}_i)\mid\mathfrak{p}_i\in\Min(R)\}.
\end{equation}
\end{theorem}
\begin{proof}
Let \(\Min(R)=\{\mathfrak{p}_1,\ldots,\mathfrak{p}_s\}\) and let \(\mathfrak{a}=\prod_{i}\mathfrak{p}_i\). Then \(\mathfrak{a}\) is contained in the nilradical of \(R\), hence \(\mathfrak{a}^N=(0)\) for some \(N\). Therefore it is clear that \(h_{\mathrm{alg}}(\varphi,R)=h_{\mathrm{alg}}(\overline{\varphi},R/\mathfrak{a}^N)\). But by Lemma~\ref{combine}
\[h_{\mathrm{alg}}(\overline{\varphi},R/\mathfrak{a}^N)=\max\{h_{\mathrm{alg}}(\overline{\varphi}_i,R/\mathfrak{p}_i)\mid \mathfrak{p}_i\in\Min(R)\}.\] 
\end{proof}
\begin{remark}
As we shall see shortly in Proposition~\ref{Minimals}, under certain conditions, when a self-map is \emph{integral}, a power of the self-map `fixes' minimal prime ideals. As a result, we can apply Theorem~\ref{irreducible components} to a power of our self-map in this case. We will obtain formulas similar to Formula~\ref{irrcomps} in Corollary~\ref{fistintegral1} and Theorem~\ref{irreducible components2}, below.
\end{remark}
\begin{lemma}\label{factored}
Let \(f:(R,\mathfrak{m})\rightarrow(S,\mathfrak{n})\) and \(g:(S,\mathfrak{n})\rightarrow(R,\mathfrak{m})\) be homomorphisms of finite length of Noetherian local rings. Then \(h_{\mathrm{alg}}(g\circ f,R)=h_{\mathrm{alg}}(f\circ g,S)\).
\end{lemma}
\begin{proof}
Before we begin the proof, it is noteworthy that by~\cite[Theorem~15.1(i), p.~116]{Matsumura2} the assumptions about \(f\) and \(g\) imply \(\dim R=\dim S\). To simplify notations we will write \(\varphi:=g\circ f\) and \(\psi:=f\circ g\). Then \(\varphi\) is a local self-map of \(R\) and \(\psi\) is a local self-map of \(S\). Moreover, by Corollary~\ref{flc}, \(\varphi\) and \(\psi\) are of finite length. Writing \(\varphi^{n+1}\) in the form \(g\circ\psi^n\circ f\) and using Corollary~\ref{proplambda}, we obtain
\[\lambda(\varphi^{n+1})\leq\lambda(g)\cdot\lambda(\psi^n)\cdot\lambda(f).\]
By symmetry we also obtain
\[\lambda(\psi^{n+1})\leq\lambda(f)\cdot\lambda(\varphi^n)\cdot\lambda(g).\]
Applying logarithm to these inequalities, then dividing by \(n+1\) and letting \(n\) go to infinity, we obtain \[h_{\mathrm{alg}}(g\circ f,R)=h_{\mathrm{alg}}(f\circ g,S).\]
\end{proof}
\begin{corollary}[Invariance]\label{isoinva}
Let \((R,\mathfrak{m})\) and \((S,\mathfrak{n})\) be two Noetherian local rings. Suppose \(f:R\rightarrow S\) is an isomorphism, and let \(\varphi\) be a self-map of of finite length of \(R\) . Then \(h_{\mathrm{alg}}(f\circ\varphi\circ f^{-1},S)=h_{\mathrm{alg}}(\varphi,R)\).
\end{corollary}
\begin{proof}
It suffices to apply Lemma~\ref{factored} to the following homomorphisms
\[\begin{tikzpicture} \matrix (m) [matrix of math nodes, row sep=1em, column sep=1.7em, text height=1ex, text depth=0.25ex] { R & & S \\}; \path[->, font=\scriptsize]
(m-1-3) edge [bend left=30] node[auto] {$ f^{-1} $} (m-1-1); \path[->, font=\scriptsize] (m-1-1) edge [bend right=-30] node[auto] {$ f\circ\varphi $} (m-1-3);
\end{tikzpicture}\]
\end{proof}
\begin{corollary}\label{phident}
Let \((R,\mathfrak{m})\) be a Noetherian local ring, let \(\varphi\) be a self-map of finite length of \(R\). Let \(\mathfrak{a}\) be an ideal of \(R\) with the property \(\varphi(\mathfrak{a})R\subset\mathfrak{a}\), and write \(\overline{\varphi}\) for both self-maps induced by \(\varphi\) on \(R/\mathfrak{a}\) and \(R/\varphi(\mathfrak{a})R\). Then
\[h_{\mathrm{alg}}(\overline{\varphi},R/\mathfrak{a})=h_{\mathrm{alg}}(\overline{\varphi},R/\varphi(\mathfrak{a})R).\]
\end{corollary}
\begin{proof}
It suffices to apply Lemma~\ref{factored} to the following homomorphisms
\[\begin{tikzpicture} \matrix (m) [matrix of math nodes, row sep=1em, column sep=0.6em, text height=1ex, text depth=0.25ex] { R/\mathfrak{a} & &\ \ \ \ \ \ R/\varphi(\mathfrak{a})R \\}; \path[->, font=\scriptsize]
(m-1-3) edge [bend left=28] node[auto] {$ \overline{\mathrm{id}} $} (m-1-1); \path[->, font=\scriptsize] (m-1-1) edge [bend right=-27] node[auto] {$ \varphi^\prime $} (m-1-3);
\end{tikzpicture}\]
where \(\varphi^\prime\) and \(\overline{\mathrm{id}}\) are canonical homomorphisms induced by \(\varphi\) and identity of \(R\), respectively.
\end{proof}
\begin{proposition}\label{Minimals}
Let \(R\) be a Noetherian local ring and let \(\varphi\) be an \emph{integral} local self-map of \(R\). Let \(^a\varphi\) be the self-map induced by \(\varphi\) on \(\Spec(R)\). Assume that all minimal prime ideals of \(R\) contain \(\ker\varphi\). Then the restriction of \(^a\varphi\) to \(\Min(R)\) is a permutation of the set \(\Min(R)\).
\end{proposition}
\begin{proof}
Let \(\overline{\varphi}\) be the map \((R/\ker\varphi)\rightarrow R\) induced by \(\varphi\). We have a commuting diagram
\[\begin{tikzpicture} 
\matrix (m) [matrix of math nodes, row sep=2.8em, column sep=2.3em, text height=2ex, text depth=0.25ex] { R & R \\ (R/\ker\varphi) & \\}; 
\path[->, font=\scriptsize]
(m-1-1) edge node[auto] {$ \varphi $} (m-1-2) 
              edge node[left] {$ \pi $}(m-2-1);
\path[right hook->, font=\scriptsize] 
(m-2-1) edge node[below right] {$ \overline{\varphi} $} (m-1-2);
\end{tikzpicture}\]
Let \(\mathfrak{q}\) be an element of \(\Min(R)\). Then by assumption \(\ker\varphi\subset\mathfrak{q}\), hence \(\pi(\mathfrak{q})\) is an element of \(\Min(R/\ker\varphi)\). Since \(\varphi\) is integral, there is an element \(\mathfrak{p}\in\Spec(R)\) such that \(\pi(\mathfrak{q})=\overline{\varphi}^{\:-1}(\mathfrak{p})\). Thus, \(\mathfrak{q}=\varphi^{-1}(\mathfrak{p})\), or equivalently \(\mathfrak{q}={^a\varphi}(\mathfrak{p})\). We claim that \(\mathfrak{p}\in\Min(R)\). If \(\mathfrak{p}\) were not a minimal prime ideal of \(R\), then it would contain a minimal prime ideal \(\mathfrak{p}^\prime\). In that case we would have \(\pi(\mathfrak{q})=\overline{\varphi}^{\:-1}(\mathfrak{p})\supseteq\overline{\varphi}^{\:-1}(\mathfrak{p}^\prime)\) and the minimality of \(\pi(\mathfrak{q})\) would force \(\overline{\varphi}^{\:-1}(\mathfrak{p}^\prime)=\pi(\mathfrak{q})\). But since \(\varphi\) is integral, there can be no inclusion between prime ideals of \(R\) lying over \(\pi(\mathfrak{q})\)~\cite[Theorem~9.3, p.~66]{Matsumura2}. This establishes our claim that \(\mathfrak{p}\in\Min(R)\). Thus, we see that \[\Min(R)\subseteq{^a\varphi}\left(\Min(R)\right).\] Now, since \(\Min(R)\) is a finite set, we must have \(\Min(R)={^a\varphi}\left(\Min(R)\right)\). Hence the restriction of \(^a\varphi\) to \(\Min(R)\) is a bijective map of the set \(\Min(R)\) to itself.
\end{proof}
\begin{corollary}\label{fistintegral1}
Let \((R,\mathfrak{m})\) be a Noetherian local ring, and let \(\varphi\) be an integral local self-map of \(R\). Assume that all minimal prime ideals of \(R\) contain \(\ker\varphi\). Let \(p\) be the smallest integer such that \(^a\varphi^p\) is the identity map on \(\Min(R)\) (see Proposition~\ref{Minimals}.) For each \(\mathfrak{p}_i\in\Min(R)\) let \(\overline{\varphi}_i\) be the self-map induced by \(\varphi^p\) on \(R/\mathfrak{p}_i\). Then
\[h_{\mathrm{alg}}(\varphi,R)=\frac{1}{p}\cdot\max\{h_{\mathrm{alg}}(\overline{\varphi}_i,R/\mathfrak{p}_i)\mid\mathfrak{p}_i\in\Min(R)\}.\]
\end{corollary}
\begin{proof}
First we should note that since \(\varphi\) is integral, it is of finite length. By Theorem~\ref{irreducible components} 
\[h_{\mathrm{alg}}(\varphi^p,R)=\max\{h_{\mathrm{alg}}(\overline{\varphi}_i,R/\mathfrak{p}_i)\mid \mathfrak{p}_i\in\Min(R)\}.\] Now by Proposition~\ref{exponent} \(h_{\mathrm{alg}}(\varphi^p,R)=p\cdot h_{\mathrm{alg}}(\varphi,R)\) and the result follows.
\end{proof}

\begin{corollary}
Let \(R\) be a Noetherian local ring and let \(\varphi\) be an integral local self-map of \(R\). Assume that all minimal prime ideals of \(R\) contain \(\ker\varphi\). Then an element \(x\in R\) belongs to a minimal prime ideal of \(R\), if and only if \(\varphi(x)\) belongs to a minimal prime ideal of \(R\).
\end{corollary}
\begin{proof}
Let \(x\) be an element of \(R\). If \(\varphi(x)\in\mathfrak{p}\) for some \(\mathfrak{p}\in\Min(R)\), then we have \(x\in\varphi^{-1}(\mathfrak{p})\) and by Proposition~\ref{Minimals} \(\varphi^{-1}(\mathfrak{p})\in\Min(R)\). Conversely, suppose \(x\in\mathfrak{q}\) for some \(\mathfrak{q}\in\Min(R)\). Then by Proposition~\ref{Minimals} there is a minimal prime ideal \(\mathfrak{p}\in\Min(R)\) such that \(\mathfrak{q}=\varphi^{-1}(\mathfrak{p})\). Hence \(\varphi(x)\in\mathfrak{p}\). 
\end{proof}
\begin{corollary}
Let \(R\) be a Noetherian local ring and let \(\varphi\) be an integral local self-map of \(R\). Assume that all minimal prime ideals of \(R\) contain \(\ker\varphi\). If \(\mathfrak{p}\not\in\Min(R)\), then \(\varphi^{-1}(\mathfrak{p})\not\in\Min(R)\).
\end{corollary}
\begin{proof}
This follows quickly from the proof of Proposition~\ref{Minimals}.
\end{proof}\begin{remark}\label{theindimap}
If \(R\) is a Noetherian local ring and \(\varphi\) is a local self-map of \(R\), then for every integer \(n\geq1\), \(\varphi\left(\ker\varphi^n\right)\subset\ker\varphi^{n-1}\subset\ker\varphi^n\). Hence \(\varphi\) induces a local self-map of \(R/\ker\varphi^n\). 
\end{remark}
\begin{proposition}\label{preporatory}
Let \((R,\mathfrak{m})\) be a Noetherian local ring, let \(\varphi\) be a self-map of of finite length of \(R\). For any integer \(n\geq1\) let \(\overline{\varphi}_n\) be the local self-map induced by \(\varphi\) on \(R/\ker\varphi^n\) (see Remark~\ref{theindimap}). Then\smallskip

\begin{compactenum}
\item[\textbf{a})] \(h_{\mathrm{alg}}(\varphi,R)=h_{\mathrm{alg}}(\overline{\varphi}_n,R/\ker\varphi^n)\).
\item[\textbf{b})] For \(n\) large enough \(\overline{\varphi}_n:R/\ker\varphi^n\rightarrow R/\ker\varphi^n\) is injective.
\item[\textbf{c})] If \(\varphi\) is integral, then so is \(\overline{\varphi}_n\) (see~\cite[Chapter~V, Proposition~2, p.~305]{Bourb}).
\end{compactenum}
\end{proposition}
\begin{proof}
\textbf{a}) We apply Corollary~\ref{phident} to the self-map \(\varphi^n\) of \(R\), taking \(\ker\varphi^n\) as the ideal \(\mathfrak{a}\) in that corollary. Since \(\varphi^n(\ker\varphi^n)R=(0)\), we obtain
\[h_{\mathrm{alg}}(\overline{\varphi}_n^n,R/\ker\varphi^n)=h_{\mathrm{alg}}(\overline{\varphi}_n^n,R/\varphi^n(\ker\varphi^n)R)=h_{\mathrm{alg}}(\varphi^n,R).\] Now the result follows from Proposition~\ref{exponent}.\par \textbf{b}) \(R\) is Noetherian, so the ascending chain \(\ker\varphi\subset\ker\varphi^2\subset\ker\varphi^3\subset\ldots\) is stationary. Let \(n_0\) be such that \(\ker\varphi^n=\ker\varphi^{n+1}\) for \(n\geq n_0\). We will show that if \(n\geq n_0\), then \(\overline{\varphi}_n:R/\ker\varphi^n\rightarrow R/\ker\varphi^n\) is injective. Let \(\overline{x}\in R/\ker\varphi^n\). Saying \(\overline{\varphi}_n(\overline{x})=0\) is equivalent to saying \(\varphi(x)\in\ker\varphi^n\), which is equivalent to saying \(x\in\ker\varphi^{n+1}\). Since \(\ker\varphi^{n+1}=\ker\varphi^n\), we see that \(x\in\ker\varphi^n\), or \(\overline{x}=0\) in \(R/\ker\varphi^n\). Thus, \(\overline{\varphi}_n\) is injective.\par \textbf{c}) We have a commutative diagram 
\[\begin{tikzpicture} 
\matrix (m) [matrix of math nodes, row sep=2.7em, column sep=2em, text height=2ex, text depth=0.25ex] { R & R \\ (R/\ker\varphi^n) & (R/\ker\varphi^n).\\}; 
\path[->, font=\scriptsize]
(m-1-1) edge node[auto] {$ \varphi $} (m-1-2) 
              edge node [left] {$ \pi $} (m-2-1) 
(m-2-1) edge node[auto] {$ \overline{\varphi}_n $} (m-2-2)
(m-1-2) edge node [auto] {$ \pi $}  (m-2-2);
\end{tikzpicture}\]
Let \(\pi(x)\in R/\ker\varphi^n\) be an arbitrary element. Since \(x\in R\) is integral over the subring \(\varphi(R)\) of \(R\), it satisfies an equation \[x^n+\varphi(a_{n-1})x^{n-1}+\ldots+\varphi(a_1)x+\varphi(a_0)=0\] with \(a_i\in R\), for \(0\leq i\leq n-1\). Applying \(\pi\) and using the commutativity of the diagram, we conclude 
\[\left(\pi(x)\right)^n+\overline{\varphi}_n\left(\pi(a_{n-1})\right)\left(\pi(x)\right)^{n-1}+\ldots+\overline{\varphi}_n\left(\pi(a_1)\right)\pi(x)+\overline{\varphi}_n\left(\pi(a_0)\right)=0.\]
Thus \(\pi(x)\) is integral over the subring \(\overline{\varphi}_n(R/\ker\varphi^n)\) of \(R/\ker\varphi^n\).
\end{proof}
\begin{theorem}\label{irreducible components2}
Let \((R,\mathfrak{m})\) be a Noetherian local ring, let \(\varphi\) be an integral local self-map of \(R\), and let \(\overline{\varphi}_n\) be the self-map induced by \(\varphi\) on \(R/\ker\varphi^n\)\emph{(see Remark~\ref{theindimap})}. Assume \(n\) is large enough so that \(\overline{\varphi}_n\) is injective \emph{(see Proposition~\ref{preporatory}-\textbf{b})}. Let \(p\) be the smallest integer such that \(^a\overline{\varphi}_n^{\:p}\) is the identity map on \(\Min(R/\ker\varphi^n)\) \emph{(see Propositions~\ref{preporatory}-\textbf{c} and~\ref{Minimals})}. If if \(\mathfrak{p}_i\in\Spec(R)\) is such that \(\overline{\mathfrak{p}}_i:=(\mathfrak{p}_i/\ker\varphi^n)\in\Min(R/\ker\varphi^n)\), then \(\varphi^p(\mathfrak{p}_i)R\subset\mathfrak{p}_i\). Moreover, if we denote the local self-map induced by \(\varphi^p\) on \(R/\mathfrak{p}_i\) by \(\overline{\varphi}_{\mathfrak{p}_i}\), then
 \[h_{\mathrm{alg}}(\varphi,R)=\displaystyle{\frac{1}{p}}\cdot\max\left\{h_{\mathrm{alg}}(\overline{\varphi}_{\mathfrak{p}_i},R/\mathfrak{p}_i)\mid(\mathfrak{p}_i/\ker\varphi^n)\in\Min(R/\ker\varphi^n)\right\}.\]
\end{theorem}
\begin{proof}
First we should note that since \(\varphi\) is integral, it is of finite length. The relation \(\varphi^p(\mathfrak{p}_i)R\subset\mathfrak{p}_i\) quickly follows from the assumption that \(^a\overline{\varphi}_n^{\:p}\) is the identity map on \(\Min(R/\ker\varphi^n)\) and the following commutative diagram
\[\begin{tikzpicture} 
\matrix (m) [matrix of math nodes, row sep=2.7em, column sep=2em, text height=2ex, text depth=0.25ex] { R & R \\ (R/\ker\varphi^n) & (R/\ker\varphi^n).\\}; 
\path[->, font=\scriptsize]
(m-1-1) edge node[auto] {$ \varphi^p $} (m-1-2) 
              edge (m-2-1) 
(m-2-1) edge node[auto] {$ \overline{\varphi}_n^{\:p} $} (m-2-2)
(m-1-2) edge (m-2-2);
\end{tikzpicture}\]
Now, by Propositions~\ref{preporatory}-\textbf{a} and~\ref{exponent} \[h_{\mathrm{alg}}(\varphi,R)=(1/p)\cdot h_{\mathrm{alg}}(\overline{\varphi}_n^{\:p},R/\ker\varphi^n),\]
and by Theorem~\ref{irreducible components} applied to the ring \(R/\ker\varphi^n\) and its self-map \(\overline{\varphi}_n^{\:p}\)
\[h_{\mathrm{alg}}(\overline{\varphi}_n^{\:p},R/\ker\varphi^n)=\]\[\max\Big\{h_{\mathrm{alg}}\Big(\overline{\varphi}_{\overline{\mathfrak{p}}_i},\frac{R/\ker\varphi^n}{\mathfrak{p}_i/\ker\varphi^n}\Big)\mid\overline{\mathfrak{p}}_i=(\mathfrak{p}_i/\ker\varphi^n)\in\Min(R/\ker\varphi^n)\Big\},\]
where we wrote \(\overline{\varphi}_{\overline{\mathfrak{p}}_i}\) for the local self-map induced by \(\overline{\varphi}_n^{\:p}\) on \((R/\ker\varphi^n)/(\mathfrak{p}_i/\ker\varphi^n)\). Finally, by applying Proposition~\ref{quotients} two times we can see
\[h_{\mathrm{alg}}\Big(\overline{\varphi}_{\overline{\mathfrak{p}}_i},\frac{R/\ker\varphi^n}{\mathfrak{p}_i/\ker\varphi^n}\Big)=h_{\mathrm{alg}}(\overline{\varphi}_{\mathfrak{p}_i},R/\mathfrak{p}_i).\] This concludes our proof.
\end{proof}
Embedded prime ideals of a ring may not behave as nicely as its minimal prime ideals under a self-map. As the next example will show, the image of a non-zero divisor under a self-map could land in an embedded prime ideal, and it is possible for an element of an embedded prime ideal to be mapped to a non-zerodivisor by a self-map.
\begin{example} Consider the polynomial ring \(K[x,y,z,w]\) over a field \(K\). Let \(\mathfrak{a}\) be the ideal  \((x^2,xy,xz,zw)\) and let \(A=K[x,y,z,w]/\mathfrak{a}\). Then \[\Ass(A) = \{(x,z), (x,w), (x,y,z)\}.\] Define a self-map \(\varphi\) of \(K[x,y,z,w]\) as follows 
\[x \stackrel{\varphi}{\mapsto} x^2;\ \  y \stackrel{\varphi}{\mapsto} y;\ \  z \stackrel{\varphi}{\mapsto} w;\ \  w\stackrel{\varphi}{\mapsto} z.\]
Since \(\varphi(\mathfrak{a})\subset\mathfrak{a}\), \(\varphi\) induces a self-map \(\overline{\varphi}\) of \(A\). 
The \(A\)-module \(\Fib A\) is finitely generated. In fact, it is generated by \(1\) and \(x\) as an \(A\)-module. Now, \(y+w\) is not a zerodivisor in \(A\) because it does not belong to any prime ideal in \(\Ass(A)\). But \(\overline{\varphi}(y+w) = y+z\) is a zerodivisor in \(A\); it is killed by \(x\), for example. On the other hand, \(y+z\) is a zerodivisor but is mapped to \(y+w\), which is not a zerodivisor.
\end{example}
\section{Finite self-maps, algebraic entropy and degree}\label{Findegr}
Another analogy between algebraic entropy $h_{\mathrm{alg}}(\varphi,R)$ and topological entropy $h_{\mathrm{top}}(\varphi,X)$ is exhibited by their relation to degree of $f$. In the topological setting, Misiurewicz and Przytycki showed in~\cite{MisPrz}, that if $f$ is a $C^1$ self-map of a smooth compact orientable manifold $M$, then $$h_{\mathrm{top}}(f,M)\geq\log|\deg(f)|.$$ In addition, for a holomorphic self-map $f$ of $\mathbb{CP}^n$, Gromov established the relation $$h_{\mathrm{top}}(f,\mathbb{CP}^n)=\log|\deg(f)|$$ in~\cite{Gromov}. Here $\deg(f)$ is the topological degree of $f$. In this section we will obtain similar relations between algebraic entropy and degree of a finite self-map (see Theorem~\ref{degree1}). First, we will make it clear what we mean by degree here.
\begin{definition}\label{degree}
Let \(R\) be a Noetherian local domain, and let \(\varphi\) be a finite self-map of \(R\). Then by \emph{degree} of \(\varphi\), \(\deg(\varphi)\), we mean the rank of the \(R\)-module \(\Fi R\).  Note that the equality \(\deg(\varphi^n)=[\deg(\varphi)]^n\) holds for all \(n\in\mathbb{N}\).
\end{definition}
\begin{lemma}\label{multiplicity3}
Let \((R,\mathfrak{m})\) be a Noetherian local domain, let \(\varphi\) be a finite local self-map of \(R\). If \(\mathfrak{q}\) is an \(\mathfrak{m}\)-primary ideal of \(R\), then for \(n\in\mathbb{N}\)
\begin{equation}\label{Samuel's Formula}
e(\varphi^n(\mathfrak{q})R)=\frac{e(\mathfrak{q})\left(\deg(\varphi)\right)^n}{[\Fi k:k]^n}.
\end{equation}
\end{lemma}
\begin{proof}
Let \(d=\dim R\). By definition of multiplicity
\begin{eqnarray}
e(\mathfrak{q},\Fin R)&=&\lim_{m\rightarrow\infty}\frac{d!}{m^d}\cdot\ell_R\Big(\frac{\Fin R}{\mathfrak{q}^m\cdot\Fin R}\Big)\nonumber\\
&=&\lim_{m\rightarrow\infty}\frac{d!}{m^d}\cdot\ell_R\Big(\Fin\big(\frac{R}{\varphi^n(\mathfrak{q}^m)R}\big)\Big)\nonumber\\
&=&\lim_{m\rightarrow\infty}\frac{d!}{m^d}\cdot\ell_R\Big(\Fin\big(\frac{R}{(\varphi^n(\mathfrak{q})R)^m}\big)\Big)\nonumber
\end{eqnarray}
(for the last equality, see, e.g.,~\cite[Exercise~1.18, p.~10]{AtMc}). Now by Proposition~\ref{product formula}
\begin{eqnarray}
\lim_{m\rightarrow\infty}\frac{d!}{m^d}\cdot\ell_R\Big(\Fin\big(\frac{R}{(\varphi^n(\mathfrak{q})R)^m}\big)\Big)&=&\lim_{m\rightarrow\infty}\frac{d!}{m^d}\cdot[\Fin k:k]\cdot\ell_R\Big(\frac{R}{(\varphi^n(\mathfrak{q})R)^m}\Big)\nonumber\\
&=&[\Fi k:k]^n\cdot\lim_{m\rightarrow\infty}\frac{d!}{m^d}\cdot\ell_R\Big(\frac{R}{(\varphi^n(\mathfrak{q})R)^m}\Big)\nonumber\\
&=&[\Fi k:k]^n\cdot e(\varphi^n(\mathfrak{q})R).\nonumber
\end{eqnarray}
So \(e(\mathfrak{q},\Fin R)=[\Fi k:k]^n\cdot e(\varphi^n(\mathfrak{q})R)\). On the other hand (see~\cite[Theorem~14.8, p.~109]{Matsumura2}) \[e(\mathfrak{q},\Fin R)=e(\mathfrak{q})\cdot\deg(\varphi^n),\]
and Formula~\ref{Samuel's Formula} quickly follows.
\end{proof}
\begin{remark}\label{on degree}
Formula~\ref{Samuel's Formula} can also be deduced from~\cite[Corollary~1, Chapter~VIII, p.~299]{ZarSam}.
\end{remark}
\begin{theorem}\label{degree1}
Let \((R,\mathfrak{m})\) be a Noetherian local domain, and let \(\varphi\) be a finite local self-map of \(R\). Then\smallskip

\begin{compactenum}
\item[\emph{\textbf{a})}] \(\log\deg(\varphi)\leq \log[\Fi k:k]+h_{\mathrm{alg}}(\varphi,R)\). 
\item[\emph{\textbf{b})}] If in addition \(R\) is Cohen-Macaulay, then \[\log\deg(\varphi)=\log[\Fi k:k]+h_{\mathrm{alg}}(\varphi,R).\]
\end{compactenum}
\end{theorem}
\begin{proof}
\textbf{a}) Consider a minimal free presentation of the \(R\)-module \(\Fin R\)
\[R^s\rightarrow R^t\rightarrow\Fin R\rightarrow0.\]
Localizing this presentation at \((0)\) we can see 
\[(\deg(\varphi))^n=\deg(\varphi^n)=\rank\Fin R\leq t=\nu(\Fin R).\] 
Now, since by Corollary~\ref{genandleng}
\[\nu(\Fin R)=[\Fi k:k]^n\cdot\lambda(\varphi^n),\] we obtain
\[(\deg(\varphi))^n\leq[\Fi k:k]^n\cdot\lambda(\varphi^n).\]
We get the desired inequality by applying logarithm, dividing by \(n\) and letting \(n\) approach infinity.\par \textbf{b}) Suppose \(R\) is Cohen-Macaulay and let \(\mathfrak{q}\) be an arbitrary parameter ideal of \(R\). Then
\[\lambda(\varphi^n)=\ell_R\left(R/\varphi^n(\mathfrak{m})R\right)\leq\ell_R\left(R/\varphi^n(\mathfrak{q})R\right).\]
Since \(R\) is Cohen-Macaulay (see, e.g.,~\cite[Theorem~17.11, p.~138]{Matsumura2})
\[\ell_R\left(R/\varphi^n(\mathfrak{q})R\right)=e(\varphi^n(\mathfrak{q})R).\]
Whence
\[\lambda(\varphi^n)\leq e(\varphi^n(\mathfrak{q})R)=\frac{e(\mathfrak{q})(\deg(\varphi))^n}{[\Fi k:k]^n},\]
where the last equality holds by Lemma~\ref{multiplicity3}. Applying logarithm, dividing by \(n\), and letting \(n\) approach infinity we obtain
\[h_{\mathrm{alg}}(\varphi,R)\leq\log\deg(\varphi)-\log[\Fi k:k].\]
This inequality together with the inequality in part \textbf{a}) yield the result.
\end{proof}

\section{Bounds for algebraic entropy}\label{Bounn}
The following definition is inspired by a similar definition given in~\cite[p.~11]{Samuel}.
\begin{definition}\label{lowup}
Let \((R,\mathfrak{m})\) be a Noetherian local ring, and let \(\varphi\) be a self-map of finite length of \(R\). We define
\begin{eqnarray}
v(\varphi)&=&\max\{k\in\mathbb{N}\mid\varphi(\mathfrak{m})R\subset\mathfrak{m}^k\},\nonumber\\ 
w(\varphi)&=&\min\{k\in\mathbb{N}\mid\mathfrak{m}^k\subset\varphi(\mathfrak{m})R\}.\nonumber
\end{eqnarray}
\end{definition}
\begin{remark}\label{comparison}
It quickly follows from this definition, that for all \(n\in\mathbb{N}\)
\[\mathfrak{m}^{w(\varphi^n)}\subset\varphi^n(\mathfrak{m})R\subset\mathfrak{m}^{v(\varphi^n)}.\]
Thus, we always have \(v(\varphi^n)\leq w(\varphi^n)\).
\end{remark}
\begin{lemma}\label{subadditive}
Let \((R,\mathfrak{m})\) be a Noetherian local ring, and let \(\varphi\) be a self-map of finite length of \(R\). Then for all \(m,n\in\mathbb{N}\) the following inequalities hold:
\begin{eqnarray}
v(\varphi^{n+m})&\geq& v(\varphi^n)\cdot v(\varphi^m),\nonumber\\
w(\varphi^{n+m})&\leq& w(\varphi^n)\cdot w(\varphi^m).\nonumber
\end{eqnarray}
\end{lemma}
\begin{proof}
First note that for any ideal \(\mathfrak{a}\) of \(R\) and any \(n\in\mathbb{N}\) we have (see~\cite[Exercise~1.18, p.~10]{AtMc}) \[\varphi(\mathfrak{a}^n)R=\left(\varphi(\mathfrak{a})R\right)^n.\] Now, we can write
\begin{eqnarray}
\varphi^{m+n}(\mathfrak{m})R&=&\varphi^m(\varphi^n(\mathfrak{m})R)R\subset\varphi^m(\mathfrak{m}^{v(\varphi^n)})R\nonumber\\
&=&\left[\varphi^m(\mathfrak{m})R\right]^{v(\varphi^n)}R\subset\mathfrak{m}^{v(\varphi^m)v(\varphi^n)}.\nonumber
\end{eqnarray}
Thus, by definition of \(v(\varphi^{m+n})\) we must have \(v(\varphi^{m+n})\geq v(\varphi^n)\cdot v(\varphi^m)\). Similarly, we can write
\begin{eqnarray}
\mathfrak{m}^{w(\varphi^m)w(\varphi^n)}&\subset&\left[\varphi^m(\mathfrak{m})R\right]^{w(\varphi^n)}=\varphi^m(\mathfrak{m}^{w(\varphi^n)})R\nonumber\\
&\subset&\varphi^m(\varphi^n(\mathfrak{m})R)R=\varphi^{m+n}(\mathfrak{m})R.\nonumber
\end{eqnarray}
Again, by definition of \(w(\varphi^{m+n})\) we must have \(w(\varphi^{m+n})\leq w(\varphi^m)\cdot w(\varphi^n)\).

\end{proof}
\begin{theorem}\label{conv2}
Let \((R,\mathfrak{m})\) be a Noetherian local ring, and let \(\varphi\) be a self-map of finite length of \(R\). Then the sequences \(\{(\log v(\varphi^n))/n\}\) and \(\{(\log w(\varphi^n))/n\}\) are both convergent. We will denote the limits of these sequences by \(v_h(\varphi)\) and \(w_h(\varphi)\), respectively.
\end{theorem}
\begin{proof}
We will apply Lemma~\ref{subadditively convergent}, taking the sequences \(\{v(\varphi^n)\}\) and \(\{w(\varphi^n)\}\) as \(\{a_n\}\) and \(\{b_n\}\) in the lemma, respectively. We verify that the conditions of the lemma are satisfied. By Lemma~\ref{subadditive} and Remark~\ref{comparison}, for every \(n\in\mathbb{N}\) we have 
\[1\leq[v_(\varphi)]^n\leq v(\varphi^n)\leq w(\varphi^n)\leq[w(\varphi)]^n.\]
Thus, condition \textbf{a}) of Lemma~\ref{subadditively convergent} is satisfied. Moreover, Lemma~\ref{subadditive} shows that condition \textbf{b}) of Lemma~\ref{subadditively convergent} is also satisfied. Hence the sequences \(\{\log(v(\varphi^n))/n\}\) and \(\{\log(w(\varphi^n))/n\}\) are both convergent. 
\end{proof}
\begin{lemma}
Let \((R,\mathfrak{m})\) be a Noetherian local ring with of dimension \(d\), and let \(\varphi\) be a self-map of finite length of \(R\). Then \[d\cdot v_h(\varphi)\leq h_{\mathrm{alg}}(\varphi,R)\leq d\cdot w_h(\varphi).\]
\end{lemma}
\begin{proof}
By Definition~\ref{lowup} we have
\[\mathfrak{m}^{w(\varphi^n)}\subset\varphi^n(\mathfrak{m})R\subset\mathfrak{m}^{v(\varphi^n)}.\] Thus
\[\ell_R(R/\mathfrak{m}^{v(\varphi^n)})\leq\lambda(\varphi^n)\leq\ell_R(R/\mathfrak{m}^{w(\varphi^n)}).\] We consider two cases: \(v(\varphi^n)\rightarrow\infty\) and \(v(\varphi^n)\not\rightarrow\infty\). In the first case, if \(v(\varphi^n)\rightarrow\infty\), then by Remark~\ref{comparison} \(w(\varphi^n)\rightarrow\infty\), as well, and for large \(n\), the lengths \(\ell_R(R/\mathfrak{m}^{v(\varphi^n)})\) and \(\ell_R(R/\mathfrak{m}^{w(\varphi^n)})\) are polynomials in \(v(\varphi^n)\) and \(w(\varphi^n)\), respectively, of precise degree \(d\), with highest degree terms \(e(\mathfrak{m})(v(\varphi^n))^d/d!\) and \(e(\mathfrak{m})(w(\varphi^n))^d/d!\). Thus, for large \(n\) we will have
\[\frac{e(\mathfrak{m})}{d!}\left(v(\varphi^n)\right)^d\leq\lambda(\varphi^n)\leq\frac{e(\mathfrak{m})}{d!}\left(w(\varphi^n)\right)^d.\]
Applying logarithm, dividing by \(n\) and letting \(n\) approach infinity, we see that
\[0\leq d\cdot v_h(\varphi)\leq h_{\mathrm{alg}}(\varphi,R)\leq d\cdot w_h(\varphi)<\infty.\] 
In the second case, when \(v(\varphi^n)\not\rightarrow\infty\), the sequence \(\{v(\varphi^n)\}\) must be bounded. Hence, there is a constant \(c\) such that \(1\leq v(\varphi^n)\leq c\). Applying logarithm, dividing by \(n\) and letting \(n\) approach infinity, we see that \(v_h(\varphi)=0\). Now, if \(w(\varphi^n)\rightarrow\infty\), then starting with the inequality 
\[1\leq\lambda(\varphi^n)\leq\ell_R(R/\mathfrak{m}^{w(\varphi^n)})\]
and repeating the same argument as before, we arrive at the desired inequality
\[v_h(\varphi)=0\leq h_{\mathrm{alg}}(\varphi)\leq d\cdot w_h(\varphi).\]
Finally if \(w(\varphi^n)\not\rightarrow\infty\), then the sequence \(\{w(\varphi^n)\}\) is also bounded and there exists a constant \(c^\prime\) such that \(1\leq w(\varphi^n)\leq c^\prime\). After applying logarithm, dividing by \(n\) and letting \(n\) approach infinity, we see that \(w_h(\varphi)=0\). Since \(v_h(\varphi)=0\) as well, the proof will be completed by showing \(h_{\mathrm{alg}}(\varphi,R)=0\). This follows from the inequality
\[1\leq\lambda(\varphi^n)\leq\ell_R(R/\mathfrak{m}^{w(\varphi^n)})\leq\ell_R(R/\mathfrak{m}^{c^\prime})\]
by applying logarithm, dividing by \(n\) and letting \(n\) approach infinity.
 \end{proof}
\begin{corollary}
Let \((R,\mathfrak{m})\) be a Noetherian local ring, and let \(\varphi\) be a self-map of finite length of \(R\). If \(w(\varphi^n)\not\rightarrow\infty\) then \(h_{\mathrm{alg}}(\varphi,R)=0\).
\end{corollary}
\appendix\section{Topological entropy}\label{appind}
A local self-map of finite length of a Noetherian local ring \(R\) induces a self-map on \(\Spec(R)\), as well as the punctured spectrum \(U_R\) of \(R\). The complexity of the induced map may be measured by the notion of \emph{topological entropy} that was introduced in~\cite{AdKoMc}. The induced map could be as simple as the identity map, even when the ring self-map is not identity. For example, when \(R\) is of positive characteristic \(p\), then the Frobenius endomorphism induces the identity map on the punctured spectrum. It is then not surprising that the topological entropy of the Frobenius endomorphism is \(0\), while its algebraic entropy is \(d\cdot\log p\), as was shown in Example~\ref{FrobEnd}. The following definitions are taken from~\cite{AdKoMc}.\par Let \(X\) be a compact topological space. For any open cover \(\mathcal{U}\) of \(X\) let 
\[N(\mathcal{U})=\min\{\mid\mathcal{V}\mid\colon \mathcal{V}\ \mathrm{is\ a\ subcover\ of\ }\mathcal{U}\}.\]
Note that since \(X\) is compact, \(N(\mathcal{U})\) is always finite. If \(\mathcal{U}\) and \(\mathcal{W}\) are two open covers of \(X\), define their \emph{join} \(\mathcal{U}\vee\mathcal{W}\) as follows
\[\mathcal{U}\vee\mathcal{W}=\{U\cap W\mid U\in\mathcal{U}, W\in\mathcal{W}\}.\]
An open cover \(\mathcal{W}\) is said to be a \emph{refinement} of a cover \(\mathcal{U}\), written \(\mathcal{U}\prec\mathcal{W}\), if every member of \(\mathcal{W}\) is a subset of some member of \(\mathcal{U}\).
\begin{definition}[\cite{AdKoMc}, p.~310]\label{topent}
Let \(\varphi:X\rightarrow X\) be a continuous self-map of a compact topological space \(X\). Then the topological entropy of \(\varphi\), \(h_{\mathrm{top}}(\varphi,X)\) is defined as follows
\[h_{\mathrm{top}}(\varphi,X)=\sup_{\mathcal{U}} h(\varphi,\mathcal{U}),\]
where
\[h(\varphi,\mathcal{U})=\lim_{n\rightarrow\infty}\frac{1}{n}\cdot\log N\left(\mathcal{U}\vee\varphi^{-1}(\mathcal{U})\vee\ldots\vee\varphi^{-(n-1)}(\mathcal{U})\right).\]
\end{definition}
The existence of the limit in Definition~\ref{topent} was established in~\cite[p.~310]{AdKoMc}. 
\begin{proposition}[\cite{AdKoMc}, Property~10, p.~310]\label{basicopen}
If \(\mathcal{U}\) and \(\mathcal{W}\) are open covers of a topological space \(X\) such that \(\mathcal{U}\prec\mathcal{W}\) and if \(\varphi\) is a continuous self-map of \(X\), then \(h(\varphi,\mathcal{U})\leq h(\varphi,\mathcal{W})\).
\end{proposition}
In applications to algebraic geometry, one can consider topological entropy for self-maps of \(\Spec(R)\), as well as self-maps of the punctured spectrum \(U_R\) of a Noetherian local ring \(R\). Note that \(\Spec(R)\) and \(U_R\) are  compact spaces (see~\cite[Proposition~1.1.4, p.~195]{GroDieu}).

\bibliographystyle{amsplain}

\end{document}